\documentclass[12pt,a4paper]{amsart}
\usepackage{graphics}
\usepackage{epsfig}
\usepackage{graphicx}
\theoremstyle{plain}
\usepackage{amssymb}
\advance\hoffset-20mm \advance\textwidth40mm

\newtheorem{theorem}{Theorem}
\newtheorem{lemma}{Lemma}
\newtheorem*{theo*}{Theorem}

\newtheorem{corollary}{Corollary}
\theoremstyle{definition}

\newtheorem*{definition*}{Definition}

\newtheorem{remark}{Remark}

\def\ad{\mathop{\rm ad}}

\def\Der{{\rm Der}}
\begin{document}
\sloppy
\title[On  finite dimensional Lie algebras of planar vector fields with rational coefficients]
{On  finite dimensional Lie algebras  \\ of  planar vector fields with rational coefficients}
\author
{ Ievgen  Makedonskyi and  Anatoliy  Petravchuk}
\address{Ievgen  Makedonskyi:
Department of Algebra and Mathematical Logic, Faculty of Mechanics and Mathematics, Kyiv
Taras Shevchenko University, 64, Volodymyrska street, 01033  Kyiv, Ukraine}
\email{makedonskyi@univ.kiev.ua , makedonskyi.e@gmail.com}
\address{Anatoliy P. Petravchuk:
Department of Algebra and Mathematical Logic, Faculty of Mechanics and Mathematics, Kyiv
Taras Shevchenko University, 64, Volodymyrska street, 01033  Kyiv, Ukraine}
\email{aptr@univ.kiev.ua , apetrav@gmail.com}
\date{\today}
\keywords{Lie Algebra, Vector Field,  Derivation, Finite Dimensional Subalgebra}
\subjclass[2000]{Primary 17B66; Secondary 17B05, 13N15}

%
\begin{abstract}
The Lie algebra of planar vector fields with coefficients from the
field of rational functions over an algebraically closed field of
characteristic zero is considered. We find all finite-dimensional Lie
algebras that can be realized as subalgebras of this algebra.
\end{abstract}
\maketitle


\section*{Introduction}
Let $\mathbb K$ be an algebraically closed field of characteristic zero and $R=\mathbb K(x, y)$
be the field of rational  functions.  Recall that a $\mathbb K$-linear mapping $D:  R\to R$ is called a $\mathbb K$-derivation if
$D(fg)=D(f)g+fD(g)$ for all $f, g\in R.$ We  denote by $\widetilde{W_2}(\mathbb K)$ the Lie algebra of all
$\mathbb K$-derivations of $R,$  this algebra is a two-dimensional vector space over $R,$ its  basis $\{ \frac{\partial}{\partial x},
\frac{\partial}{\partial y}\}$  will be called standard.  In geometric terms, a derivation $D$ is a vector field with rational coefficients and $\widetilde{W_2}(\mathbb K)$ is the Lie algebra of all vector fields on $\mathbb K^{2}$  with rational coefficients. The Lie algebra $\widetilde{W_2}(\mathbb K)$
is closely connected with the automorphism group ${\rm Aut} (R)$ of the field $R$ (for example if $D$ is a locally nilpotent derivation of $R,$   then $\exp D$ is an automorphism of $R$). The group ${\rm Aut} (R)$ was intensively studied  by many authors (see, for example \cite{DI}). A question about finite subgroups of ${\rm Aut} (R)$ is of special interest, the description of such subgroups was  recently completed
 by I.Dolgachev and V.Iskovskikh \cite{DI}. So, it is of interest to study finite dimensional subalgebras of the Lie algebra
 $\Der (R)=\widetilde{W_2}(\mathbb K)$ which corresponds in some sense to ${\rm Aut} (R)$.

 In this paper, we give  a description of finite dimensional subalgebras of $\widetilde{W_2}(\mathbb K)$ up to isomorphism as Lie algebras using only algebraic tools, so some results can  be transferred  to Lie algebras of derivations of extensions of fields. Such a description can  also be obtained (over the field of complex numbers) using analytical and geometric methods; it can be deduced from results of  S.Lie (see \cite{Lie}, S.71-73). There are many papers devoted to such subalgebras,  see for example  \cite{Draisma}, \cite{Olver},  \cite{Olver1}, \cite{PBNL}, \cite{Nesterenko}, \cite{J1}, \cite{Post}.
 The main result of the paper is Theorem \ref{main} where all types  of finite dimensional subalgebras of $\widetilde{W_2}(\mathbb K)$  are listed. From this description one can easily obtain  all possible types of finite dimensional subalgebras of the Lie algebra ${W_2}(\mathbb K)=\Der  \mathbb K[x, y]$ (up to isomorphism as Lie algebras). We do not consider the problem of finding all inequivalent realizations of such finite dimensional algebras up to automorphisms of the field $R=\mathbb K(x, y).$  Note  that this problem in the case $\mathbb K=\mathbb C$ can be easily solved using results of of S.Lie   \cite{Lie} (see also \cite{Olver} and  \cite{Nesterenko}).

We use standard  notations, the ground field $\mathbb K$ is  algebraically closed of characteristic zero (some results are valid for any field of characteristic $0$). If $D_{1}, \ldots , D_{n}$ are elements of $\widetilde{W_2}(\mathbb K)$, then we denote by $\mathbb K\langle D_{1}, \ldots , D_{n}\rangle$ or simply $\langle D_{1}, \ldots , D_{n}\rangle$
 the linear span of elements $D_{1}, \ldots , D_{n}$ over the field $\mathbb K.$ The field $\mathbb K(x, y)$ of rational functions  will be denoted by $R,$ every nonzero $\mathbb K$-subspace of $\widetilde{W_2}(\mathbb K)$ has rank $1$ or $2$ over $R$ as a system  of elements of the two-dimensional vector space $\widetilde{W_2}(\mathbb K)$ over $R.$

\section{Preliminaries}

\begin{lemma}\label{basic}
Suppose that $D_1, D_2 \in \widetilde{W}_2(\mathbb{K}).$ Then

$(1)$ For any $a, b \in R$ it holds
$[aD_1,  bD_2 ]=ab[ D_1, D_2]+aD_1(b)D_2-bD_2(a)D_{1}.$

$(2)$ If $D_1, D_2$ are linearly independent over $R$ and  $D_1(c)=D_2(c)=0$ for some $c\in R,$ then $c \in \mathbb{K}.$
\end{lemma}
\begin{proof}
$1.$ Straightforward calculation.

$2.$ Note that $\frac{\partial}{\partial x}, $  $\frac{\partial}{\partial y} \in \widetilde{W}_2(\mathbb{K})$ are linear combinations of $ D_1$ and $ D_2$ with coefficients in $R.$  Then $\frac{\partial}{\partial x}(c)=\frac{\partial}{\partial x}(c)=0$ which implies $c\in \mathbb K.$
\end{proof}

\begin{lemma}\label{rank1}
Let $L$ be a finite dimensional subalgebra of the Lie algebra  $\widetilde{W}_{2}(\mathbb{K}).$ If $L$ is of rank $1$
over $R,$  then there exists  an element  $D_{1}\in \widetilde{W}_{2}(\mathbb{K})$ such that  $L$ is one of the following algebras:

$(1)$  $L=\langle D_{1}, a_{1}D_{1}, \ldots , a_{n}D_{1}\rangle $  for some  $a_{i}\in R$ such that  $D_{1}(a_{i})=0$ for all $i.$ The algebra  $L$ is abelian.

$(2)$ $L=\langle D_{1}, a_{1}D_{1}, \ldots , a_{n-1}D_{1}, bD_{1}\rangle $  for some  $a_{i}, b\in R$ such that $D_{1}(a_{i})=0$  for all $i,$  $D_{1}(b)=-1.$   $L$ is metabelian.

$(3)$  $L=\langle D_{1}, - a^{2}D_{1}, -2aD_{1} \rangle $  for some $a\in R$ with  $D_{1}(a)=1.$ The algebra  $L$ is isomorphic to $sl_{2}(\mathbb K ).$
\end{lemma}

\begin{proof}
Replacing the polynomial ring $\mathbb K[x, y]$ by the field $R=\mathbb K(x, y)$  in the proof of Theorem 1 in \cite{AMP} one can show that a finite dimensional subalgebra of rank $1$  over $R$ from $\widetilde{W_2}(\mathbb K)$ is either abelian, or metabelian of the form $L=\langle b\rangle \rightthreetimes A, [b, a]=a$  for all $a\in A$ with  abelian $A,$   or $L\simeq sl_{2}(\mathbb K).$  Consider all these cases. If $L$ is abelian with $\mathbb K$-basis $\{ D_{1}, a_{1}D_{1}, \ldots , a_{n}D_{1}\}$ then $[D_{1}, a_{i}D_{1}]=0=D_{1}(a_{i})D_{1}$ for all $i.$ Hence $D_{1}(a_{i})=0$ for all $i$ and $L$ is of type 1. Let $L=\langle b\rangle \rightthreetimes A$ with abelian subalgebra $A=\{ D_{1}, a_{1}D_{1}, \ldots , a_{n-1}D_{1}\}.$ Then as above $D_{1}(a_{i})=0$ for all $i$ and since $[bD_{1}, D_{1}]=D_{1}$ we get $D_{1}(b)=-1.$ Thus $L$ is of type 2. Finally, let $L\simeq sl_{2}(\mathbb K).$ Choose the standard basis $\{ e, f, h\}$ of $L$ over $\mathbb K.$ Without loss of generality we may put $e=D_{1}, f=bD_{1}, h=aD_{1}$ for some $a, b\in R.$ Then
$$[aD_{1}, D_{1}]=2D_{1}, [aD_{1}, bD_{1}]=-2bD_{1},  [D_{1}, bD_{1}]=aD_{1},$$ so using Lemma \ref{basic} we get from the first equality that $D_{1}(a)=-2.$ The second equality implies $aD_{1}(b)+2b=-2b$ and therefore $D_{1}(b)=-4b/a.$ The third equality yields $D_{1}(b)=a.$ So, $a=-4b/a$ and $a^{2}=-4b,$ i.e. $b=-a^{2}/4.$ We get the basis $\{ D_{1}, -a^{2}/4D_{1}, aD_{1}\}$ of the algebra $L.$  Replacing here $a$  by $-a/2$ we obtain a basis
$\{ D_{1}, -a^{2}D_{1}, -2aD_{1}\}$  where $D_{1}(a)=1.$
\end{proof}

\begin{remark}\label{rem1}
One can easily point out realizations for Lie algebras from the previous Lemma:
1. $D_{1}=\frac{\partial}{\partial x}, a_{i}=y^{i}, i=1, \ldots ,n; $ 2. $ D_{1}=\frac{\partial}{\partial x}, a_{i}=y^{i}, i=1, \ldots , n-1, b=-x; $ 3. $ D_{1}=\frac{\partial}{\partial x}, a=x.$
\end{remark}
\begin{lemma}\label{ideals}
Let $L\not= 0$ be a  finite dimensional solvable subalgebra of the Lie algebra $\widetilde{W}_{2}(\mathbb{K})$ and
let $\langle D_{1}\rangle$  be its arbitrary one-dimensional ideal. Then

$(1) $ The set $I= RD_{1}\cap L$ is an ideal of $L.$

$(2)$  $\dim L/I\leq 2$ and if $\dim L/I=2,$ then the quotient algebra $L/I$ is nonabelian.

$(3)$  If $\dim L\geq 5,$ then the ideal $I$  contains all ideals of rank 1 over $R$ from $L.$
\end{lemma}
\begin{proof}
1. Take any element $D\in L.$ Since $\langle D_{1}\rangle$ is an ideal of $L$ we have
$[D, D_{1}]=\lambda D_{1}$ for some $\lambda \in \mathbb K$ depending on $D.$ Then for any element $aD_{1}\in I$ it holds
$$[D, aD_{1}]=D(a)D_{1}+a[D, D_{1}]=(D(a)+\lambda a)D_{1}\in I.$$
Therefore $I$ is an ideal of $L.$

2.  We can obviously assume that $I\not= L.$ Choose a one-dimensional ideal $\langle D_{2}+I\rangle$ of the quotient algebra $L/I.$
As $D_{2}\not\in I$ the elements $D_{1}, D_{2}$ are linearly independent over $R.$ It suffices to show that the ideal $J=I+\langle D_{2}\rangle$ of the algebra $L$ is of codimension $\leq 1$ in $L.$ Take arbitrary elements
$D_{3}=a_{3}D_{1}+b_{3}D_{2},$ $ D_{4}=a_{4}D_{1}+b_{4}D_{2}$ with $ a_{3}, a_{4}, b_{3}, b_{4} \in R$
from the set $L\setminus J.$  Since
$$ [D_{1}, D_{3}]=D_{1}(a_{3})D_{1}+D_{1}(b_{3})D_{2}+b_{3}\lambda D_{1}\in  \langle D_{1}\rangle$$
(here $[D_{2}, D_{1}]=\lambda D_{1}$) we get $D_{1}(b_{3})=0.$  Analogously from the relation $[D_{2}, D_{3}]\in J$ we have $D_{2}(b_{3})=c_{3}\in \mathbb K.$ Similar  calculations yield $D_{1}(b_{4})=0, D_{2}(b_{4})=c_{4}\in \mathbb K.$ It can be easily shown that $c_{3}\not= 0,  c_{4}\not= 0.$ Indeed, let to the contrary $c_{3}=0.$ Then the equalities $D_{1}(b_{3})=0, D_{2}(b_{3})=c_{3}=0$ imply by Lemma \ref{basic} that $b_{3}\in \mathbb K.$ This means that $a_{3}D_{1}\in L$ and as $a_{3}D_{1}\in I$ we get $D_{3}\in J.$ The latter  contradicts to the choice of $D_{3}.$  Analogously one can show that $c_{4}\not= 0.$ Consider the element $c_{4}D_{3}-c_{3}D_{4}$ of $L$ and write it in the form
$$(c_{4}a_{3}-c_{3}a_{4})D_{1}+(c_{4}b_{3}-c_{3}b_{4})D_{2}=r_{1}D_{1}+r_{2}D_{2}.$$
  Straightforward calculation shows that $D_{2}(r_{2})=0.$  As also  $D_{1}(r_{2})=0,$ the element $r_{2}$ belongs to $\mathbb K$ by Lemma \ref{basic}. Therefore $c_{4}D_{3}-c_{3}D_{4}\in J$ and $D_{3}, D_{4}$ are linearly dependent over $J,$ i.e. $\dim L/J\leq 1.$

Now let $\dim L/I=2$ and $\{D_{2}+I, aD_{1}+bD_{2}+I\}$ be a basis of $L/I.$ Suppose that $L/I$ is abelian. Then $[D_{2}, aD_{1}+bD_{2}]\in I$ and therefore $D_{2}(b)=0.$  From the relation $[aD_{1}+bD_{2}, D_{1}]\in
\langle D_{1}\rangle$ it follows that $D_{1}(b)=0.$ But then  Lemma \ref{basic} yields $b\in \mathbb K$  which implies
 $aD_{1}\in L.$ This means that $aD_{1}\in I$ and $aD_{1}+bD_{2}\in I+\langle D_{2}\rangle .$ The latter is impossible because the elements $D_{2}$ and $aD_{1}+bD_{2}$ are linearly independent over $I.$  This contradiction shows that $L/I$ is nonabelian.

3.  Finally, let $\dim L\geq 5,$  $I= RD_{1}\cap L$ and $T =RD_{2}\cap L$  for some  ideals  $\langle D_{1}\rangle$ and $\langle D_{2}\rangle .$  Suppose that  elements $D_{1}$ and $D_{2}$ linearly independent over $R.$      Since    $\dim L/I\leq 2 $
and $  dim L/T\leq 2$ (by the above proven)  and  $I\cap T=0$ we get  $\dim L\leq 4$ which contradicts to our assumption. Thus,  $I$ contains all ideals  of rank $1$ over $R.$
\end{proof}
We need also some elementary properties of  rational functions in a single variable.
These properties seem to be known but having no reference we supply  them with complete proofs.
For a rational function $\varphi \in \mathbb{K}(t)$ we will  denote $\varphi'=\frac{d \varphi}{d t}$. If $p(t)\in \mathbb K[t]$ is an irreducible polynomial, then  ${\rm ord}_{p}\varphi $ denotes as usually the integer $\alpha$ from the decomposition of
$\varphi$ into the product of the form  $\varphi =p^{\alpha} \psi ,$ where neither numerator nor the denominator of $\psi$ is divisible by $p.$

\begin{lemma} \label{rat_mulem}
Let $\mathbb{K}$ be an algebraically closed field of characteristic zero. Then:

$(1)$  If $\varphi(t) \in \mathbb{K}(t) \backslash \mathbb{K}$,
then there does not exist any function $\psi \in \mathbb{K}(t)$ such that $\psi'=\frac{\varphi'}{\varphi}$.

$(2)$  Let  $\varphi, \psi \in \mathbb{K}(t) \backslash \mathbb{K}$ be such functions that
$\mu \varphi' \psi-\varphi \psi'=0 $
for some $\mu \in \mathbb{K}$. Then $\mu \in \mathbb{Q}$, $\mu=\frac{m}{n}$, and $\varphi^m=c \psi^n$ for some $c \in \mathbb{K}$. Moreover,  there exists $\theta \in \mathbb{K}(t)$ such that $\varphi=c_{1} \theta^s$, $\psi=c_{2} \theta^t$
for some $c_{1}, c_{2}  \in \mathbb{K}$, $s,t \in \mathbb{Z}$.
\end{lemma}
\begin{proof}
1. Suppose  on the contrary that there exists  $\psi \in \mathbb{K}(t)$ such that $\psi'=\frac{\varphi'}{\varphi}$.
Let $p \in \mathbb{K}[t]$ be an irreducible polynomial such that ${\rm ord}_p (\varphi) \neq 0$.
Put $\alpha ={\rm ord}_p (\varphi)$. Then $\varphi=p^{\alpha} q$ and $\varphi'=\alpha p^{\alpha-1}p'q+p^{\alpha} q'$.
Therefore
\[\frac{\varphi'}{\varphi}=\frac{\alpha p' p^{\alpha-1} q+p^{\alpha} q'}{p^{\alpha} q}=\frac{\alpha  q p'+p q'}{p q}. \]
Since ${\rm ord}_p(\alpha qp'+pq')=0$  it holds  ${\rm ord}_p\left( \frac{\varphi'}{\varphi} \right)=-1$ (note that  ${\rm ord}_p(q)=0$).
 Now put $\beta={\rm ord}_p(\psi)$, $\psi=p^{\beta} r$. Then $\psi'=\beta p' p^{\beta-1} r+p^{\beta} r'$. If $\beta=0$, then
$\psi'= r'$ and  ${\rm ord}_p(\psi')={\rm ord}_p(r') \geq 0$. Suppose that   $\beta \neq 0.$ Then
\[{\rm ord}_p(\psi')={\rm ord}_p(\beta p' p^{\beta-1} r+p^{\beta} r')={\rm ord}_p(\beta p' p^{\beta-1} r)=\beta-1. \]
Therefore in any  case ${\rm ord}_p{\psi'} \neq -1$, which contradicts  to the equality ${\rm ord}_p\left( \frac{\varphi'}{\varphi} \right)=-1$. Hence there does not  exist such a polynomial $\psi$ that  $\psi'=\frac{\varphi'}{\varphi}.$

2.  Take any functions $\varphi , \psi $  from  $\mathbb{K}(t) \backslash \mathbb{K}$ satisfying  the condition
 \begin{equation}
\mu \varphi' \psi-\varphi \psi'=0. \label{rat_mu}
\end{equation}
It can be easily shown that there exists a  point $t_0 \in \mathbb{K}$  such that ${\rm ord}_{t-t_0}\varphi \neq 0$ ( because the field $\mathbb{K}$ is algebraically closed).
Without loss of generality we can assume that the field $\mathbb{K}(t)$ is embedded to the field $\mathbb{K}((t))$ of Laurent series at the point $t_0$.
Put
\[\varphi=\sum_{i=m}^\infty \alpha_i (t-t_0)^i, \psi=\sum_{i=n}^\infty \beta_i (t-t_0)^i, {\rm where }~ m, n \in \mathbb{Z}, \alpha_m\beta_n \neq 0.\]
Since  $ord_{t-t_0}\varphi \neq 0,$ it holds $m \neq 0$,
We can assume that $\alpha_m= \beta_n = 1$, because the equation (\ref{rat_mu}) is homogeneous.
Computing  coefficients at $t^{m+n-1}$ in  both sides of the equation (\ref{rat_mu}) we obtain $\mu m=n$. Therefore $\mu=n/m \in \mathbb{Q}$. Further,
\[\left( \frac{\varphi^n}{\psi^m} \right)'=\frac{n \varphi^{n-1}\varphi'\psi^m-m \varphi^{n}\psi^{m-1}\psi'}{\psi^{2m}}=
\frac{\varphi^{n-1}\psi^{m-1}(n \varphi' \psi-m\varphi \psi')}{\psi^{2m}}=0,\]
because $n \varphi' \psi-m\varphi \psi'=m(\mu \varphi' \psi-\varphi \psi')=0$. Hence,  $\frac{\varphi^n}{\psi^m} \in \mathbb{K}$
i.e.  $\varphi^n=c \psi^m$ for some  $c \in \mathbb{K}$.

The functions  $\varphi$ and $\psi$ can be written as products of  irreducible factors with (nonzero) integer powers
\[\varphi=\prod_{i=1}^s u_i^{k_i},    \   \psi=\prod_{j=1}^k v_i^{l_i}.\]
Using the equality $\varphi^n=c \psi^m$ we get   $k=s$ and after renumbering the factors  we can  write down
$u_i=\gamma_i v_i$ for some $\gamma_i \in \mathbb{K}$. Hence we have:
\[\left(\prod_{i=1}^k u_i^{k_i} \right)^n=c\left(\prod_{i=1}^k (\gamma_i u_i)^{l_i} \right)^m.\]
This equality implies that $n k_i=m l_i$ for all $i=1, \dots , k$.
Denote $d=\gcd (m,n)$ and $m=m_1 d$, $n=n_1 d$. We obtain  equalities $n_1 d k_i=m_1 d l_i$, $i=1, \dots , k$,
and therefore $n_1 k_i=m_1 l_i$. Since $\gcd(m_1, n_1)=1$ we obtain  that $l_i$ is divisible by $n_1,$
$k_i$ is divisible by $m_1$, $i=1, \dots , k.$ Denote  $\frac{l_{i}}{n_{1}}=\frac{k_{i}}{m_{1}}=r_{i}$  and  $\theta =\prod_{i=1}^s u_i^{r_i}$.
Then $\varphi=\theta^{m_1}$, $c_1 \psi=\theta^{n_1}$ for some $c_1 \in \mathbb{K}^\ast$.
This completes the proof of  Lemma.
\end{proof}

\begin{lemma} \label{properties}
Let $D_1$ and $D_2$ be elements of $\widetilde{W_2}(\mathbb{K})$ linearly  independent over $R$  such that
$[D_2, D_1]=\nu D_1$ for some $\nu \in \mathbb K.$  Let $b_1, b_2$  be linearly independent over $\mathbb K$ elements of $R \backslash \mathbb{K}$ such that $D_1(b_i)=0, i=1,2.$  Then:

$(1)$  If  $[D_2, b_i D_1]=\lambda_i b_i D_1$ for some $\lambda_i \in \mathbb{K}, i=1,2,$
then $\lambda_1 \neq \lambda_2$. If $\lambda_1 \neq \nu$, then $\frac{\lambda_2-\nu}{\lambda_1 -\nu} \in \mathbb{Q}.$

$(2)$  If  $[D_2, b_1D_1]=\lambda b_1D_1, [D_2, b_2 D_1]=\lambda b_2D_1 + b_1D_1 $ for some $\lambda \in \mathbb{K},$ then $\lambda =\nu$.
\end{lemma}
\begin{proof} 1.  Using the condition  $[D_2, b_i D_1]=\lambda_i b_i D_1$  we get
 \begin{equation}
 D_2(b_i)=(\lambda_i-\nu )b_i, i=1,2. \label{Db}
 \end{equation}
 Suppose that  $\lambda_1=\lambda_2.$
Then $D_2 \left (\frac{b_1}{b_2}\right )=\frac{D_{2}(b_1)b_2-b_1D_{2}(b_2)}{b_2^2}=0$. Besides, $D_{1}(\frac{b_{1}}{b_{2}})=0$ by
 conditions of Lemma.
 Then using linear independence of  elements $D_1, D_2$ we obtain  by Lemma \ref{basic} the inclusion  $ \frac{b_1}{b_2} \in \mathbb{K}$.
 The latter is impossible because of  linear independence of elements $b_1, b_2$ over $\mathbb{K}$.
 Hence  $\lambda_1 \neq \lambda_2$.

 Let now $\lambda _{1}\not= \nu .$ Since $b_1, b_2 \in R \backslash \mathbb{K}$, the  subfield $\ker (D_1)$ of $R$ is of  transcendence degree $1$ over $\mathbb K$
 (it is obvious that this degree cannot be equal to $2$). Hence $\ker D_{1}$  is generated by a single element
 (see, for example, \cite{Schinzel}, Th.3 and \cite{NN}). Denote this element by $\theta$. Then $b_1=\varphi_1(\theta)$, $b_2=\varphi_2(\theta)$ for some rational functions $ \varphi _{1}(t) ,$  $ \varphi _{2}(t) \in \mathbb K(t).$
 Using the relation  $[D_2, D_1]=\nu D_1$ we see  that $D_2(\theta) \in \ker(D_1)$.
 Denote also $D_2(\theta)=f(\theta)$, $f \in \mathbb{K} (t)$.
 The conditions  (\ref{Db}) imply
$$ \varphi '_{1}(\theta)f(\theta)=(\lambda _{1}-\nu )\varphi _{1}(t), \    \varphi '_{2}(\theta)f(\theta)=(\lambda _{2}-\nu )\varphi _{2}(\theta).$$
 Since $\varphi_i$ are not constants and $\lambda_1-\nu \neq 0$ we have:
$$\varphi_1\varphi_2'-\mu \varphi_1'\varphi_2=0, \  \mbox{\rm where } \ \mu=\frac{\lambda_2-\nu}{\lambda_1-\nu}.$$
 Now Lemma \ref{rat_mulem}  yields  the inclusion  $\mu \in \mathbb{Q}$.

 2. By the condition (\ref{Db}) of Lemma we have
 \begin{equation}
 D_2(b_1)=(\lambda  -\nu)b_1, ~D_2(b_2)=(\lambda -\nu)b_2+b_1. \label{Db1}
 \end{equation}

 As above  we can show  that $b_1=\psi_1(\theta)$, $b_2=\psi_2(\theta)$, where $\theta$ is a  generator of the subfield $\ker D_1$ and  $D_2(\theta)=g(\theta)$ for some rational functions  $\psi_1, \psi_2, g \in \mathbb{K}(t)$. Using (\ref{Db1}) one can easily show  that
 \begin{equation}
 \psi_1'g=(\lambda- \nu)\psi_1, ~\psi_2'g=(\lambda-\nu)\psi_2+\psi_1. \label{phi1}
 \end{equation}
 Since $b_1 \in R \backslash \mathbb{K}$ it holds  $\psi_1' \neq 0$.
 The equality (\ref{phi1}) implies the next relations
 \begin{equation}
 \frac{\psi_1'}{\psi_1}=
 \frac{(\lambda-\nu)\psi_2'}{(\lambda-\nu)\psi_2 +\psi_1}=\left(\frac{(\lambda-\nu)\psi_2}{\psi_1} \right)' \label{cont}
\end{equation}
 (note that $(\lambda-\nu)\psi_2 +\psi_1 \neq 0$ because $\psi_1$ and $\psi_2$ are linearly independent over $\mathbb{K}$).
 But the relation (\ref{cont})  is impossible if $\lambda \neq \nu$ by Lemma \ref{rat_mulem}.
 This contradiction shows that  $\lambda =\nu$.
 \end{proof}

 The next statement can be easily deduced from the theorem of S.Lie about solvable Lie algebras.
\begin{lemma}\label{Lie_conseq}
 Let $V$ be a finite dimensional vector space over the field $\mathbb K$ and $T, S$ be linear operators on $V.$
 If $[T, S]=S,$  then the operator $S$ is nilpotent.

\end{lemma}
\section{Finite dimensional solvable subalgebras of $\widetilde{W_2}(\mathbb{K})$}
  \begin{lemma}\label{solv1}
  Let $L$ be a  finite dimensional solvable subalgebra of  rank $2$ over $R$ of   $\widetilde{W_2}(\mathbb{K})$
  and let  $\langle D_{1}\rangle$ be its arbitrary one dimensional ideal. Denote $I=RD_{1}\cap L.$
  If  the ideal $I$ is abelian,   then there exists an element  $ D_{2}\in L\backslash I$ such that $L$ is one of the following algebras:

 {\rm  (1)}  $L=\langle D_{1}, aD_{1},  \dots , (a^{n}/n!)D_{1}, D_{2} \rangle$, where $a\in R$ such that
  $D_{1}(a)=0, D_{2}(a)=1, [D_{2}, D_{1}]=\lambda D_{1}$ and $\lambda =0$ or $\lambda =1, $ $ n\geq 1.$ If $n=0,$ we put $L=\langle D_{1}, D_{2}\rangle .$

  {\rm  (2)}   $L=\langle D_{1}, a_{1}D_{1}, \ldots , a_{n}D_{1}, D_{2} \rangle ,$  where $a_{i}\in R,  [D_{2}, D_{1}]=D_{1},
  D_{1}(a_{i})=0,  D_{2}(a_{i})=\beta m_{i}a_{i},  m_{i}\in \mathbb Z \ \  \mbox{for all} \  i, \beta \in \mathbb K^{\star}, m_{i}\not= m_{j}  \ \ \mbox{for}\  i\not= j, $ $ n\geq 1. $

 {\rm  (3)}    $L=\langle D_{1}, aD_{1},  \dots , (a^{n}/n!)D_{1}, D_{2}, bD_{1}+aD_{2} \rangle$, where  $a, b \in R$
 such that  $D_{1}(a)=0, D_{1}(b)=\beta , \beta \in \mathbb K,$ $ [D_{2}, D_{1}]=0, , D_{2}(a)=1, D_{2}(b)=(n+1)\gamma a^{n},  \gamma \in \mathbb K, $ $ n\geq 1$ (if $n=0$  we put $L=\langle D_{1}, D_{2}, bD_{1}+aD_{2}\rangle $).
  \end{lemma}
  \begin{proof}
The set $I=RD_{1}\cap L$ is an ideal of $L$ by Lemma \ref{ideals}. We can write  $I=\langle D_{1}, a_{1}D_{1}, \ldots , a_{n}D_{1}\rangle$ for some elements $a_{i}\in R$ and $n\geq 1$ (if $n=0$ we put  $I=\langle D_{1}\rangle$). Since the ideal $I$ is abelian  we have $D_{1}(a_{i})=0, \ i=1, \ldots , n.$ We consider two cases depending on $\dim L/I$ (recall that $\dim L/I\leq 2$ by Lemma \ref{ideals}).

{\underline {Case 1.}} $\dim L/I=1.$ Take any element $D_{2}\in L\backslash I.$ As $\langle D_{1}\rangle$ is an ideal of $L$ we have $[D_{2}, D_{1}]=\nu D_{1}$ for some $\nu \in \mathbb K.$ The elements $D_{1}$ and $D_{2}$ are linearly independent over $R$ by the choice of the ideal $I.$  First, let the linear operator $\ad D_{2}$ have the only eigenvalue $\nu$ on the vector space $I$ (recall that $[D_{2}, D_{1}]=\nu D_{1}$). If $aD_{1}, bD_{1}\in I$ are eigenvectors of $\ad D_{2},$ i.e. $      [D_{2}, aD_{1}]=\nu aD_{1},$ $  [D_{2}, bD_{1}]=\nu bD_{1},$ then the  elements $aD_{1}, bD_{1}$ are linearly dependent over $\mathbb K$ by Lemma \ref{properties}. Hence $D_{1}$ is the unique eigenvector of $\ad D_{2}$ on $I$  (up to multiplication by a nonzero scalar).  But then the  linear operator $\ad D_{2}$ has a Jordan basis in $I$ of the form
$\{ D_{1}, a_{1}D_{1}, \ldots , a_{n}D_{1}\},$  $a_{i}\in R$ such that
$$  [D_{2}, a_{i}D_{1}]=\nu a_{i}D_{1}+a_{i-1}D_{1}, i=1, \ldots , n, [D_{2}, D_{1}]=\nu D_{1}    $$
(in case $n=0$ we have $I=\langle D_{1}\rangle $).
The last relations imply the equalities $D_{2}(a_{i})=a_{i-1}, i=2, \ldots , n,$ $ D_{2}(a_{1})=1.$  Denoting  $a=a_{1}$ we have $D_{2}(a_{2}-a^{2}/2!)=0$ and taking into account the relation $D_{1}(a_{2}-a^{2}/2!)=0$  we see by Lemma \ref{basic} that $a_{2}-a^{2}/2!\in \mathbb K.$ But then without loss of generality we can take $a_{2}=a^{2}/2!.$  Analogously $D_{2}(a_{3}-a^{3}/3!)=a_{2}-a_{2}=0$ and $D_{1}(a_{3}-a^{3}/3!)=0,$  so we can put $a_{3}=a^{3}/3!.$
Repeating these considerations we get a $\mathbb K$-basis $\{ D_{1}, aD_{1}, \ldots, (a^{n}/n!)D_{1}\}$ of the ideal $I$
(recall that $I=\langle D_{1}\rangle$ in case $n=0$).
The algebra Lie $L$ is of type 1 because we always can assume  that $\nu =0$ or $\nu =1$ choosing a convenient multiple of the element $D_{2}.$

Now let $\ad D_{2}$ have on $I$ at least two different eigenvalues. Our aim is to show that $\ad D_{2}$ is a diagonalizable operator on $I.$ Denote by $I(\lambda )$ the root space of $\ad D_{2}$ corresponding to the eigenvalue $\lambda , \lambda \not= \nu .$ Since $\ad D_{2}$ has on $I(\lambda )$ the only eigenvalue $\lambda $ it follows from the previous considerations that $\ad D_{2}$ has on $I(\lambda )$ a Jordan basis such that the matrix of $\ad D_{2}$ in this basis is a single Jordan block. Therefore if $\dim I(\lambda )>1$ then there exist elements $aD_{1}, bD_{1}\in I$ such that
$$[D_{2}, aD_{1}]=\lambda aD_{1},   [D_{2}, bD_{1}]=\lambda bD_{1}+aD_{1}.$$
The latter is impossible by Lemma \ref{properties} and therefore $\dim I(\lambda )=1.$  Choosing any element $D_{1}'\in I$ with property $[D_{2}, D_{1}']=\lambda D_{1}'$ instead of the element $D_{1}$ and using Lemma \ref{properties} we can analogously show that $\dim I(\nu )=1,$  where $I(\nu )$ is the root space corresponding to the eigenvalue $\nu$ of $\ad D_{2}$ on $I.$  Therefore  all the root spaces are one-dimensional and $\ad D_{2}$ is diagonalizable  on $I.$

Since at least one of the eigenvalues of $\ad D_{2}$ on $I$ is nonzero we can choose elements $D_{1}$ and $D_{2}$ in such a way that
$$[D_{2}, D_{1}]=D_{1}, I=\langle D_{1}, a_{1}D_{1}, \dots , a_{n}D_{1}\rangle ,$$ where $[D_{2}, a_{i}D_{1}]=\lambda _{i}a_{i}D_{1}, \lambda _{i}\not= \lambda _{j}$ if $ i\not= j$ and $ \lambda _{i} \not= 1, i=1, \ldots , n.$

Applying Lemma \ref{properties} (with $\nu =1$) we can easily show that $\frac{\lambda _{i}-1}{\lambda _{1}-1}=\tau _{i}\in \mathbb Q,$ $i=2, \ldots , n.$  Denote $\tau _{i}=\frac{k_{i}}{l_{i}}, k_{i}, l_{i}\in \mathbb Z, i=2, \ldots , n.$
If $l$ is the least common multiple of $l_{2}, \ldots , l_{n},$ then one can write $\tau _{i}=\frac{m_{i}}{l}$ and $\lambda _{i}=m_{i}\beta +1,$ where $\beta =\frac{\lambda _{1}-1}{l}$ (note that $\lambda _{i}-1=\tau _{i}(\lambda _{1}-1)$).
Thus, $L$ is an algebra of type 2 of Lemma.

{\underline {Case 2.}} $\dim L/I=2.$ The quotient algebra $L/I$ is nonabelian by Lemma \ref{ideals}, so it contains a noncentral one-dimensional ideal $\langle D_{2}+I\rangle.$ Then there exists an element $bD_{1}+cD_{2}\in L$ such that
$$ [bD_{1}+cD_{2}+I, D_{2}+I]=D_{2}+I.     $$
This means that $[bD_{1}+cD_{2}, D_{2}]=D_{2}+gD_{1}$ for some element $gD_{1}\in I.$ Since the ideal $I$ is abelian it is obvious that $\ad D_{2}=\ad (D_{2}+gD_{1})$ on the vector space $I$ over $\mathbb K.$ We obtain the following relation for linear operators on $I:$
$$[\ad (bD_{1}+cD_{2}), \ad D_{2}]=\ad (D_{2}+gD_{1})=\ad D_{2}.   $$
But then $\ad D_{2}$ acts nilpotently on $I$ by Lemma \ref{Lie_conseq}.
 In case $\dim I=1$ we get (after direct calculations) the Lie algebra of type 3 with $n=0.$  Let $\dim I\geq 2.$ Since $[D_{2}, D_{1}]=0$ one can easily show (using Lemma \ref{ideals}) that  the ideal $I$ can be written in the form $I=\langle D_{1}, aD_{1}, \ldots , (a^{n}/n!)D_{1}\rangle$ for some $a\in R, D_{2}(a)=1, n\geq 1.$

Returning now to the above mentioned element $bD_{1}+cD_{2}\in L$ we see that
$$ [D_{1}, bD_{1}+cD_{2}]=D_{1}(b)D_{1}+D_{1}(c)D_{2}\in \langle D_{1}\rangle    $$
and therefore $D_{1}(c)=0, D_{1}(b)\in \mathbb K.$  Further, from the equality $$[D_{2}, bD_{1}+cD_{2}]=D_{2}(b)D_{1}+D_{2}(c)D_{2}\in I+\langle D_{2}\rangle .$$ we obtain  $D_{2}(c)=\gamma \in \mathbb K,$
$D_{2}(b)\in \langle 1, a, a^{2}/2!, \ldots , a^{n}/n!\rangle.$  From the relations $D_{2}(c)=\gamma \in \mathbb K$ and $ D_{2}(a)=1$ it follows that $D_{2}(\gamma a-c)=0.$ Then  Lemma \ref{basic} yields  $\gamma a-c\in \mathbb K,$ i.e. $c=\gamma a+b$ for some $\gamma , \beta \in \mathbb K.$

The element $D_{3}=\gamma ^{-1}(bD_{1}+cD_{2}-\beta D_{2})$ of the algebra $L$ can be written in the form
$D_{3}=b_{1}D_{1}+aD_{2}$ for some $b_{1}\in R.$  As $D_{2}(b_{1})\in \langle 1, a, a^{2}/2!, \ldots , a^{n}/n!\rangle$ we can substract from $b_{1}D_{1}+aD_{2}$ a suitable linear combination of the elements $D_{1}, aD_{1}, a^{2}/2!D_{1}, \ldots , a^{n}/n!D_{1}$ and assume without loss of generality that $D_{2}(b_{1})=(n+1)\gamma a^{n}$ for some $\gamma \in \mathbb K.$  Denoting $b=b_{1},$ $\beta=D_{1}(b)\in \mathbb K$ we see that  $L$ is of type 3 of this Lemma.

 \end{proof}

\begin{remark}\label{rem2}
For each type of Lie algebras from Lemma \ref{solv1} one can easily point out a realization:

1. $\lambda =0, D_{1}=\frac{\partial}{\partial x}, D_{2}=\frac{\partial}{\partial y}, a=y.$
 $\lambda =1, D_{1}=\frac{\partial}{\partial x}, D_{2}=\frac{\partial}{\partial y}-x\frac{\partial}{\partial x}, a=y.$

2. $D_{1}=\frac{\partial}{\partial x}, a_{i}=y^{m_{i}}, D_{2}=\beta y\frac{\partial}{\partial y}-x\frac{\partial}{\partial x}, \beta \in \mathbb K.$

3. $D_{1}=\frac{\partial}{\partial x}, D_{2}=\frac{\partial}{\partial y}, $ $a=y, b=\beta x+\gamma y^{n+1}, \beta , \gamma \in \mathbb K.$
\end {remark}
  \begin{lemma}\label{solv2}
Let $L$ be a subalgebra of  $\widetilde{W_2}(\mathbb{K})$ satisfying all the conditions of the previous Lemma with the exception of  that the ideal $I$ is abelian. If  $I$ is nonabelian, then there exist elements $D_{1}\in I, D_{2}\in L\backslash I$ such that $L$ is one of the following algebras:

$(1)$  $L=\langle D_{1}, aD_{1}, \ldots , (a^{n-1}/(n-1)!)D_{1}, bD_{1}, D_{2}\rangle$, where $a, b\in R$  such that  $ D_{1}(a)=0, D_{2}(a)=1, D_{1}(b)=-1, D_{2}(b)=0, [D_{2}, D_{1}]=0.$

$(2)$  $L=\langle D_{1}, a_{1}D_{1}, \dots , a_{n-1}D_{1}, bD_{1}, D_{2}  \rangle$, where $a_{i}, b \in R$ such that $[D_{2}, D_{1}]=D_{1}, D_{1}(a_{i})=0, D_{1}(b)=-1, D_{2}(b)=-b, D_{2}(a_{i})=\beta m_{i}a_{i} $ for some $  m_{i}\in \mathbb Z, \beta \in \mathbb{K}^{\star} $ and $  m_{i}\not= m_{j} \  \mbox{if} \ i\not= j.$

$(3)$  $L=\langle D_{1}, aD_{1}, \ldots , (a^{n-1}/(n-1)!)D_{1}, (v-\alpha a^{n})D_{1}, D_{2}, (-\beta v+\gamma (a^{n}/n!))D_{1}-aD_{2}\rangle ,$ where $a, v\in R$ such that $[D_{1}, D_{2}]=0, D_{1}(a)=0, D_{2}(a)=1, D_{1}(v)=-1, D_{2}(v)=0, \alpha , \beta \in \mathbb K,$ and  $ \gamma=\alpha (\beta  -n).$
\end{lemma}
\begin{proof}
Let $\langle D_{1}\rangle $ be the one-dimensional ideal of $L$ lying in $I.$ The ideal $I$ has by Lemma \ref{rank1} a basis over $\mathbb K$ of the form $\{ D_{1}, a_{1}D_{1}, \ldots , a_{n-1}D_{1}, bD_{1}\},$  where $D_{1}(a_{i})=0, D_{1}(b)=-1, i=1, \ldots , n-1$
(for $n=0$ we put  $I=\langle D_{1}, bD_{1}\rangle$ with $D_{1}(b)=-1$). Suppose that $n=0,$ i.e. $\dim I=2.$  If $\dim L/I=1,$ then
$L=\langle D_{1}, bD_{1}, D_{2}\rangle$ is of type 2 or 3. If $\dim L/I=2,$ then $L/I$ is nonabelian by Lemma \ref{ideals}
and taking into account that $L/I$ is nonabelian we have $L=I\oplus J$ for nonabelian ideal $J$ of dimension $2.$ Then $L$ is of type 3. So we may assume that $\dim I\geq 3.$
As in the previous Lemma we divide the proof into following cases:

{\underline {Case 1.}}  $\dim L/I=1.$  Take any element $D_{2}\in L\backslash I.$ Then $[D_{2}, bD_{1}]=\lambda bD_{1}+cD_{1},$ where $cD_{1}\in I'=[I, I]$ because $\dim L/I'=2$ and $\langle bD_{1}+I'\rangle$ is a one-dimensional ideal of $L/I'.$  If $\lambda \not= 0,$ then we may assume without loss of generality that $\lambda =1,$ and then
$$ [\ad D_{2}, \ad (bD_{1})]=\ad (bD_{1}+cD_{1})=\ad (bD_{1})$$ on $I'$
because $I'$ is an abelian ideal of $L.$ But then the linear operator $\ad (bD_{1})$ acts nilpotently on $I'$ by Lemma \ref{Lie_conseq}. The latter is impossible and therefore $\lambda =0.$ This means that $L/I'$ is an abelian Lie algebra of dimension $2.$
As $[D_{2}, bD_{1}]=cD_{1}$ for some element $cD_{1}\in I'$  we get $[D_{2}+cD_{1}, bD_{1}]=0$ (recall that  $[bD_{1}, cD_{1}]=cD_{1}$ for all $cD_{1}\in I'$).  So, we can choose the element $D_{2}$ in such a way that $[D_{2}, bD_{1}]=0.$ If the linear operator $\ad D_{2}$ has on $I'=\langle D_{1}, \ldots , a_{n-1}D_{1}\rangle$  at least two different eigenvalues, then there exists by Lemma \ref{properties} a basis $\{  D_{1}, \ldots , a_{n-1}D_{1}\}$ of $I'$ such that $D_{2}(a_{i})=m_{i}\beta a_{i}, $ for some $ m_{i}\in \mathbb Z, \beta \in {\mathbb K}^{\star}, $  $m_{i}\not= m_{j}$ if $i\not= j, [D_{2}, D_{1}]=D_{1}.$ Then from the relation $[D_{2}, bD_{1}]=0$ it follows $D_{2}(b)=-b.$
The algebra $L$ is of type 2 of Lemma.

Now let $\ad D_{2}$ have the only eigenvalue $\mu $ on $I'.$  If $\mu =0,$ then $L$ is obviously the Lie algebra of type 1 of Lemma. Let $\mu \not= 0.$ Taking a suitable multiple of $D_{2}$ we may assume that $\mu =1.$ Then replacing  the element $D_{2}$ by the element $D_{2}-bD_{1}$ we get the case $\mu =0$ and $L$ is again of type 1 of Lemma.

{\underline {Case 2.}}  $\dim L/I=2.$
As in the case 1 take a one-dimensional ideal $\langle D_1 \rangle$ of $L$
lying in $I'$ and a basis of $I$ of the form $\lbrace D_1,a_1D_1, \dots a_{n-1}D_1, bD_1\rbrace$
such that $D_1(a_i)=0$, $D_1(b)=-1$, $i=0, \dots n-1$.
Let $\langle D_2 + I \rangle$ be the one-dimensional ideal of the nonabelian quotient algebra $L/I$.
Accordingly to  Case 1 the algebra $\langle D_2 \rangle + I$ is of
type 1 or type 2 of this Lemma. Let us show that  $\langle D_2 \rangle + I$ is of type 1
of this Lemma, i. e. the linear operator $\ad D_2$ acts nilpotently on $I'$.
Really since $\langle bD_1 + I' \rangle$ is an ideal of the algebra $L/I'$
and $\ad (bD_1)$ acts on $I'$ as the identity operator the ideal $\langle bD_1 + I' \rangle$
lies in the center of $L/I'$ (because of Lemma  \ref{Lie_conseq}),
i. e. $[D, bD_1] \in I'$ for any element $D \in L.$
Take any element $cD_1 + dD_2 \in L \backslash I$ such that $[cD_1 + dD_2, D_2]=D_2 +rD_1$
for some element $rD_1 \in I$. The element $rD_1$ can be written in the form
$rD_1=\mu b D_1 +r_1D_1$, where $\mu \in \mathbb{K}$, $r_1D_1 \in I'$.
But then we obtain
\[[cD_1+bD_2, D_2 +\mu b D_1]=(D_2 +\mu b D_1)+ r_2D_1\]
for some element $r_2D_1 \in I'$
The latter means that $\ad (D_2+\mu b D_1)$ acts nilpotently on $I'$ (by Lemma \ref{Lie_conseq}).
Replacing the element $D_2$ by the element  $D_2+\mu b D_1$ we can
assume without loss of generality that $\ad D_2$ is a nilpotent linear operator on
$I'$. So, the subalgebra $\langle D_2 \rangle + I$ is of type 1 of this Lemma
and hence  $I'+ \langle D_2 \rangle $ can be written in the form
\[I' + \langle D_2\rangle =\langle D_1, aD_1, \dots \frac{a^{n-1}}{(n-1)!}D_1, D_2, \rangle\]
where $[D_2, D_1]=0$, $D_1(a)=0$, $D_2(a)=1$.

Further, it follows from the  above mentioned equality
\begin{equation}\label{comd2}
[cD_1+dD_2,D_2]=D_2 +r_2 D_1
\end{equation}
that $D_2(d)=-1$.
Analogously we obtain $D_1(d)=0$, $D_1(c)=\beta_1 \in \mathbb{K}$ from the relation
$[cD_1+dD_2,D_1] \in \langle D_1 \rangle$.
Since $D_2(a)=1$ and $D_2(d)=-1$ we have $D_2(a+d)=0$. Taking into account the equality
$D_1(a+d)=0$  we obtain by Lemma \ref{basic} that $a+d =\alpha_1 \in \mathbb{K}$. But then
$d=-a + \alpha_1$ and without loss of generality we can choose $cD_1-aD_2$ instead of the element
$cD_1+dD_2$.

Since $[D_2,bD_1]\in I'$ (as we have proved before) we see that
\[D_2(b)=\alpha_0 +\alpha_1 a + \dots + \alpha_{n-1} \frac{a^{n-1}}{(n-1)!}\]
for some $\alpha_i \in \mathbb{K}$.
Put $v=b-\alpha_0 a-\alpha_1 \frac{a^2}{2!} - \dots - \alpha_{n-1} \frac{a^{n}}{n!}$.
Then $D_1(v)=D_1(b)=-1, D_2(v)=0$.
Substracting the element $(\alpha_0 a+\alpha_1 \frac{a^2}{2!} + \dots + \alpha_{n-2} \frac{a^{n-1}}{(n-1)!})D_1 \in I'$
from the element $bD_1$ we can assume without loss of generality that $b=v - \alpha _{n-1} \frac{a^n}{n!}$
for some $\alpha _{n-1}  \in \mathbb{K}$. Then $D_1(b)=-1$, $D_2(b)=\alpha _{n-1} \frac{a^{n-1}}{(n-1)!}$.
Further, recall that for the basic element $cD_1-aD_2$ we have $D_1(c)=\beta_1 \in \mathbb{K}$.

Rewriting  the relation \ref{comd2} in the form
$[cD_{1}-aD_{2}, D_{2}]=D_{2}+r_{2}D_{1}$ we  obtain that
\[D_2(c)=\gamma_0+\gamma_1 a + \dots + \gamma_{n-1} \frac{a^{n-1}}{(n-1)!}   \ \ \mbox{for some} \ \  \gamma _{i}\in \mathbb K,  \  i=1, \ldots  ,n-1. \ \  \]
Substracting the element $(\gamma_0 a+\gamma_1 \frac{a^2}{2!} + \dots + \gamma_{n-2} \frac{a^{n-1}}{(n-1)!})D_1 \in I'$
from the element $cD_1-aD_2$ we may assume without loss of generality that $D_2(c)=\gamma_{n-1} \frac{a^{n-1}}{(n-1)!}$.
Suppose that $\beta_1=D_1(c) \neq 0$. Since $D_1(\beta_1^{-1}c+v -\beta_1^{-1}\gamma_{n-1} \frac{a^{n}}{n!})=0$
and $D_2(\beta_1^{-1}c+v -\beta_1^{-1}\gamma_{n-1} \frac{a^{n}}{n!})=
\beta_1^{-1}\gamma_{n-1} \frac{a^{n-1}}{(n-1)!}-\beta_1^{-1}\gamma_{n-1} \frac{a^{n-1}}{(n-1)!}=0$
we have by Lemma 1 that $\beta_1^{-1}c+v -\beta_1^{-1}\gamma_{n-1} \frac{a^{n}}{n!}= \nu$ for some $ \nu \in \mathbb{K}.$
Substracting the element $\nu D_1 \in I'$ from the element $cD_1+aD_2$ we may assume that $\nu =0$.
Then we obtain $c=-\beta_1 v+ \gamma_{n-1} \frac{a^{n}}{n!}$.
Denoting  $\alpha _{n-1}$ by $\alpha ,$ $\gamma_{n-1}$ by $\gamma$ and $\beta _{1}$ by $\beta$
 we obtain a basis of $L$ of the form:
$$  \{ D_1, aD_1, \dots , \frac{a^{n-1}}{(n-1)!}D_1, (v-\alpha \frac{a^{n}}{n!})D_1, D_2,
(-\beta v + \gamma \frac{a^{n}}{n!}) D_1-aD_2)\}$$
(here $D_{1}(a)=0, D_{1}(v)=-1, D_{2}(a)=1, D_{2}(v)=0$).
Now suppose that  $\beta =D_1(c)=0$. Since $D_2(c)=\gamma  \frac{a^{n-1}}{(n-1)!}$
we see that for the element $c_1=c-\gamma  \frac{a^{n}}{n!}$ it holds $D_1(c)=\beta=0$,
$D_2(c)=0$. So by Lemma  \ref{basic} we obtain $c-\gamma  \frac{a^{n}}{n!} = \nu_2$ for some  $\nu _{2}  \in \mathbb{K}$.
Substracting the element $\nu_2 D_1$ from $cD_1 + aD_2$ we may assume that $\nu_2=0$.
So we have that $c=\gamma  \frac{a^{n}}{n!}$ i.e. the basis of $L$ is of the same form as in case $\beta \not =0.
$

Now consider the product
$ [(v-\alpha a^{n}/n1)D_{1}, (\beta v+\gamma a^{n}/n!)D_{1}-aD_{2}].$ This product equals to $(-\alpha  \beta +\gamma +n\alpha )D_{1}$ and belongs to $I'.$  Hence $-\alpha  \beta +\gamma +n\alpha  =0$ and $\gamma =\alpha (\beta -n).$  We see that $L$ is of type 3 of Lemma.
\end{proof}
\begin{remark}\label{rem3}
There exist  realizations for all types of Lie algebras from Lemma \ref{solv2}:

1. $ D_{1}=\frac{\partial}{\partial x}, D_{2}=\frac{\partial}{\partial y}, $ $ a=y, b=-x$

2. $ D_{1}=\frac{\partial}{\partial x}, D_{2}=\beta y\frac{\partial}{\partial y}-x\frac{\partial}{\partial x}, $ $ a=y, b=-x, a_{i}=y^{m_{i}},.$

3. $ D_{1}=\frac{\partial}{\partial x}, D_{2}=\frac{\partial}{\partial y}, $ $ a=y, f=-x$

\end{remark}
The next three corollaries can be easily proved by using results of Lemmas \ref{rank1}, \ref{solv1} and \ref{solv2}.

\begin{corollary}\label{nilpotent}
Let $L$ be a finite dimensional nilpotent subalgebra of $\widetilde{W_2}(\mathbb{K}).$ Then there exist elements $D_{1}, D_{2}\in L$ linearly independent over $R$ such that $L$ is one of the following algebras:

{\rm (1)}  $L=\langle D_{1}, a_{1}D_{1}, \ldots, a_{n}D_{1}\rangle, $ for some $a_{i}\in R$  such that $D_{1}(a_{i})=0 , i=1, \ldots , n.$

{\rm (2)} $L=\langle D_{1}, D_{2}\rangle , [D_{1}, D_{2}]=0.$

{\rm (3)}  $L=\langle D_{1}, aD_{1}, \ldots , (a^{n}/n!)D_{1}, D_{2}\rangle $ for some $a\in R$ such that $ D_{1}(a)=0, D_{2}(a)=1,$ $[D_{1}, D_{2}]=0.$
\end{corollary}
\begin{corollary}\label{decompose}
Let $L$ be a finite dimensional solvable subalgebra of $\widetilde{W_2}(\mathbb{K}).$
 If $L$  is nonabelian and decomposable into a direct sum of proper ideals, then $L=A\oplus B,$ where $A$ is a nonabelian ideal of dimension $2$ and $B$ is either a one-dimensional ideal or a  two-dimensional nonabelian ideal of $L.$
\end{corollary}
\begin{corollary}\label{quotient}
Let $L$ be a finite dimensional solvable subalgebra of $\widetilde{W_2}(\mathbb{K}).$ If $L$ is nonabelian, then $\dim L/L'\leq 2.$
\end{corollary}

\section{Nonsolvable subalgebras of $\widetilde{W_2}(\mathbb K)$}

\begin{lemma} \label{semisimple}
If $L$ is a finite dimensional  semisimple subalgebra of the Lie algebra
$\widetilde{W_2}(\mathbb K),$  then $L$ is isomorphic to   $sl_{2}(\mathbb K)$ or
$sl_{3}(\mathbb K),$ or $sl_{2}(\mathbb K)\oplus sl_{2}(\mathbb K).$
\end{lemma}
\begin{proof}
If $L$ is of rank $1$ (as a system of vectors) over $R$, then $L \simeq sl_2(\mathbb{K})$ by Lemma \ref{rank1}. So, we can assume
 that $L$ is of rank $2$ over $R$. Fix a Cartan subalgebra $H$ of the algebra $L$,
 a basis $\pi$ of the system $\Delta$ of roots which correspond to $H$ and let $\Delta^+$ be the
 set of positive roots relatively to the ordering on $\Delta$. Consider the triangular decomposition
 \[L=N_+ + H + N_-,  \  N_+=\oplus_{\alpha_i >0} L_{\alpha_i},  \  N_-=\oplus_{\alpha_i <0} L_{\alpha_i}\]
 and the Borel subalgebra $B=H + N_+$ of $L.$  If the subalgebra $N_+$ is abelian, then $L$ is a direct sum
 $L=L_1 \oplus \dots \oplus L_k$ of ideals isomorphic to $sl_2(\mathbb{K})$ (see, for example \cite{Hum}). Then $B$ is a direct sum
 $B=B_1 \oplus \dots \oplus B_k$ of Borel subalgebras of $L_i \simeq sl_2 (\mathbb{K})$
 and using Corollary \ref{decompose}  we see that either $L =L_1 \simeq sl_2 (\mathbb{K})$
 or $L =L_1 \oplus L_2 \simeq sl_2 (\mathbb{K}) \oplus sl_2 (\mathbb{K})$.

 Now, let the subalgebra $N_+$ be nonabelian. Since $N_+$ is nilpotent  it is indecomposable
 into a direct sum of nonzero ideals by Corollary \ref{nilpotent}. But then the algebra $L$ is also indecomposable into
 a direct sum of proper ideals and hence is simple. By Corollary \ref{quotient} we have relations:
 \[\dim B/B'=\dim B/N=\dim H \leq 2.\]
 Therefore, if $N_+$ is nonabelian, then $\dim H=2$ and $L$ is a simple Lie algebra
 of one of the types $A_2$, $B_2$ or $G_2$.
 First suppose that  $L$ is of type $G_{2}.$ Then the subalgebra  $N_+$ from  its triangular decomposition has nonabelian  derived subalgebra  $[N_+,N_+].$ The latter is impossible (see Corollary \ref{nilpotent}) and hence $L$ cannot be of type $G_{2}.$

 Further, let us show that $L$ is not of type $B_{2}.$
 Fix a Cartan subalgebra $H$ of $L$ and a basis $\{\alpha , \beta \}$  of the  root system $\Delta .$  Then the subalgebra
 $N_{+}$ has the basis $\{ e_{\alpha},  \ e_{\beta}, \  e_{\alpha +\beta}, \ e_{\alpha +2\beta}\}.$  It follows
 from Corollary \ref{nilpotent}  that
 $e_{\alpha +\beta}=f\cdot e_{\alpha +2\beta}$ for some element $f\in R.$  Consider the element $\sigma _{\alpha}$
 of the Weyl group of the root system $\Delta$  acting by the rule $\sigma _{\alpha}(\gamma )=\gamma -
 \frac{2(\gamma , \alpha )}{(\alpha , \alpha )}\alpha ,$ where $\gamma $ is an arbitrary root from $\Delta .$
Then $\{ -\alpha , \beta +\alpha . \beta , \alpha +2\beta \}$ are  positive roots relatively to the new basis $\{ \sigma _{\alpha}(\alpha ), \sigma _{\alpha}(\beta )\}.$ The subalgebra $\langle e_{-\alpha }, e_{\beta +\alpha } , e_{\beta} ,
e_{\alpha +2\beta }\rangle$ is nilpotent and by Corollary \ref{nilpotent}  it holds $e_{\beta }=g\cdot e_{\alpha +2\beta}$ for some
$g\in R.$ Analogously one can show that $e_{\alpha}=h\cdot e_{\alpha +2\beta}$ for some $h\in R.$
Three relations with coefficients $f, g, h$  obtained above imply that all elements from the basis of $N_{+}$ are multiple to one of
them and hence the subalgebra $N_{+}$ is abelian by Lemma \ref{rank1}.  This is impossible and obtained contradiction shows that $L$ is not of type $B_{2}.$  Thus, $L$ is of type $A_{2}.$
\end{proof}

\begin{lemma}\label{Levi}
Let $L$ be a finite dimensional nonsolvable subalgebra of $\widetilde{W_2}(\mathbb K)$ whose Levi factor
 is  either of type  $A_{2}$ or of type  $A_{1}\times A_{1}.$ Then $L$ is semisimple of type
$A_{2}$ or   of type $A_{1}\times A_{1}$ respectively.
\end{lemma}
\begin{proof}
Let $S=S(L)$ be the solvable radical of $L.$ By Theorem of Levi-Maltsev $L=L_{1}\rightthreetimes S,$ where
$L_{1}$ is a Levi factor of $L.$ First suppose that $L_{1}$ is of type $A_{2}.$ Let us fix a Cartan subalgebra $H$
of $L_{1}$  and  the  root system $\Delta $ corresponding to $H.$  Consider the triangular decomposition
\begin{equation}\label{eq1}
 L=N_{-}+H+N_{+}
\end{equation}
of $L_{1}$ relatively to  $H$ and $ \Delta .$ Since the subalgebra $N_{+}$ is nonabelian (this follows from the multiplication law in algebras of  type $A_{2}$)  it contains by Corollary \ref{nilpotent} elements $D_{1}$ and $D_{2}$, linearly independent over $R $ such that $[D_{1}, D_{2}]=0.$
Consider $S$ as an $L_{1}$-module and take the older vector $D\in S$ relatively to the decompostion (\ref{eq1}). Then we have
\begin{equation}\label{eq2}
  [D_{1}, D]=0, \ \ [D_{2}, D]=0.
\end{equation}
If we write  $D=aD_{1}+bD_{2}$ for some $a, b \in R,$  then from the previous relation we get
 $$D_{1}(a)=0,  D_{1}(b)=0, D_{2}(a)=0 \ \  \mbox{and} \ \  D_{2}(b)=0.$$
 Lemma \ref{basic}  yields now that $a, b \in \mathbb K,$ i.e. $D\in L_{1}.$ As $L_{1}\cap S=0$ we obtain  $S=0$ and therefore $L=L_{1}$ is a simple Lie algebra of type $A_{2}.$

Let now $L_{1}$  be of type $A_{1}\times A_{1}.$ Write  $L_{1}=G_{1}\oplus G_{2},$ where $G_{i}\simeq sl_{2}(\mathbb K)$  and fix  Cartan subalgebras  $H_{1}\subset G_{1}, H_{2}\subset G_{2}.$
Consider any triangular decompositions $$ G_{1}=N_{1+}+H_{1}+N_{1-}, \  G_{2}=N_{2+}+H_{2}+N_{2-}$$ relatively to $H_{1}$ and $H_{2}.$  Take any nonzero element $D_{1}\in N_{1+}.$ Then at least one of the  subalgebras $N_{1-}, N_{2+}, N_{2-}$  contains a nonzero element $D_{2}$ such that $D_{1}$ and $D_{2}$ are linearly independent over $R.$
Really, in other case $H_{1}=[N_{1+}, N_{1-}]$ and $H_{2}=[N_{2+}, N_{2-}]$ lie also in $RD_{1}$ and therefore $L=G_{1}\oplus G_{2}\subset RD_{1}$  which is impossible by Lemma \ref{rank1}.
It is easily shown that the two-dimensional abelian subalgebra $N_{+}=\langle D_{1}, D_{2}\rangle$  is a part of triangular decomposition $ L=N_{+}+H+N_{-}$  of $L$ relatively to the Cartan subalgebra $H=H_{1}\oplus H_{2}.$
Choosing as above the older vector in $S$ relatively to $N_{+}$ and repeating the considerations from the case $L_{1}\simeq A_{2}$ we get $S=0,$ i.e. $L$ is semisimple of type $A_{1}\times A_{1}.$
\end{proof}

 \begin{lemma}\label{unsolvable}
 Let $L$ be a nonsolvable finite dimensional subalgebra of $\widetilde{W}_2(\mathbb{K})$.
 Then $L$ is isomorphic to one of the following algebras:

 $(1)$  $sl_3(\mathbb{K}).$

 $(2)$ $sl_2(\mathbb{K})$ or  $sl_2(\mathbb{K}) \oplus sl_2(\mathbb{K}).$

 $(3)$  $sl_2(\mathbb{K}) \rightthreetimes V_m$, where $V_m$ is the irreducible module
 over $sl_2(\mathbb{K})$ of dimension $m+1$, $m=0, 1, \dots .$

 $(4)$  $gl_2(\mathbb{K}) \rightthreetimes V_m$, where $V_m$ is the irreducible module
 over $gl_2(\mathbb{K})$ of dimension $m+1$, $m=0, 1, \dots$ and  nonzero central
 elements of $gl_2(\mathbb{K})$ act on $V_m$ as nonzero scalars.
 \end{lemma}

 \begin{proof}
 Let $S$ be the  solvable radical of $L$ and $L_1$ be a Levi factor of the algebra $L$.
 We can consider only the case $S \neq 0$ because of Lemma \ref{semisimple}.
 It follows from  Lemma \ref{Levi}  that $L_1 \simeq sl_2(\mathbb{K})$.
Choose  a Cartan subalgebra $H$ of the algebra $L_{1}$ and  a triangular decomposition
$L_{1}=N_{+}+H+N_{-}$  of $L_{1}.$

 {\underline {Case 1.}} $\dim S=1$ or $\dim S=2.$   If $\dim S=1$, then $L=L_1 \oplus S$ is a sum of two ideals
 and $L \simeq sl_2(\mathbb{K}) \oplus V_0$, where $V_0$ is a one-dimensional module over
 $sl_2(\mathbb{K})$. The algebra $L$ is of type 4 with $m=0$.
 Suppose that  $\dim S=2$. If $S$ is a nonabelian ideal of $L$, then $L$ is a direct sum of
 ideals $L=L_1 \oplus S$. Since $S=\langle w \rangle \rightthreetimes \langle v_0 \rangle$ for some elements $w, v_{0}\in S$, then $L \simeq gl_2(\mathbb{K}) \rightthreetimes \langle v_0 \rangle$ is
 of type (5) with $m=0$ because $L_{1}\oplus \langle w\rangle \simeq gl_{2}(\mathbb K).$  Let $S$ be abelian. Suppose that $S$ is a reducible
 module. Then $S=S_1 \oplus S_2$ is a direct sum of $L_1$-modules of dimension $1$ over $\mathbb{K}$.
 Take the Borel subalgebra   $B=H+N_{+}$   of $L_1.$ Then the  subalgebra $B \oplus S_1 \oplus S_2$ of $L$ is
 solvable of dimension $4$. But such an algebra does not exist by Lemmas \ref{solv1} and \ref{solv2}. This
 contradiction  shows that $S$ is irreducible and $L \simeq sl_2(\mathbb{K}) \rightthreetimes V_1$, where
 $V_1$ is of dimension 2 over $\mathbb{K}$. The algebra  $L$ is of type 4. Further, we will  assume that
 $\dim S \geq 3$.

  {\underline {Case 2.}}  $S$ is abelian (of dimension $\geq 3$).  Let us show that $S$ is an irreducible module over $L_{1}.$ Assume to the contrary that $S$ is reducible.
  If $S$ is a sum of one-dimensional submodules over $L_{1},$
 then $L=L_{1}\oplus S$ is a direct sum of ideals.  Its subalgebra $B+S$ is solvable, nonabelian and decomposable into direct sum of subalgebras $B\oplus S.$ The latter is impossible by Corollary \ref{decompose}.
 So we can assume  $S=S_1 \oplus S_2$ where $S_{1}, S_{2}$ are $L_{1}$-submodules, $\dim S_1\geq 2$
 and $S_1$ is irreducible  (note that $S_1$ and $S_2$ are ideals of $L$ because
 $S$ is abelian). Let  $D_{2} \in N_+$  be a nonzero element. Then the subalgebra $M=\langle D_2\rangle + S$ is nonabelian, nilpotent
 and $\dim M/[M, M]\leq 2$ by Corollary \ref{quotient}.
  On the other hand, since  $[M,M]=[D_2,S_1] \oplus [D_2,S_2],$  $\dim S_i/[D_2, S_i]\geq 1$, $i=1,2$ (because $\ad D_{2}$ acts nilpotently
  on $S_{i}$) we have
   $$\dim M/[M,M]=\dim \langle D_{2}\rangle +\dim S_1/[D_2, S_1] +\dim S_2/[D_2, S_2]\geq 3.$$
  The latter contadicts to Corollary \ref{quotient} and hence  $S$ is a
 simple $L_1$-module. It is obvious that   $L$ is of type 4. Note that  the subalgebra
 $M=\langle D_2\rangle + S$ is of the form $$\langle D_2, D_1, a D_1,\dots \frac{a^k}{k!} D_1 \rangle,  \ [D_2,D_1]=0, \
 D_1(a)=0, \ D_2(a)=1.$$

   {\underline {Case 3.}}  $S$ is  a nilpotent (nonabelian) ideal. Then by Corollary \ref{nilpotent} there exist elements $D_1, D_2 \in S$
 such that $$S=\langle D_2, D_1, a D_1, \ldots , (a^{k}/k!) D_1 \rangle ,  \ [D_2,D_1]=0, \
 D_1(a)=0, \ D_2(a)=1, \ \dim S \geq 3.$$
  Therefore $\langle D_1 \rangle=S^{k-1}$ and
 $\langle D_1 \rangle$ is an  ideal of $L$.
 Using Lemma \ref{ideals} we see that $R D_1 \cap L$ is an ideal of $L$ and therefore
 $L_1 \rightthreetimes \langle D_1, a D_1,\dots \frac{a^k}{k!} D_1 \rangle$ is a subalgebra
 of $L$. This subalgebra  has the abelian decomposable ideal $\langle D_1, a D_1,\dots \frac{a^k}{k!} D_1 \rangle .$ This is impossible by the  Case 1 and therefore the Case 3 is impossible.

 {\underline {Case 4.}}  $S$ is solvable (nonnilpotent). The $L_1$-submodule $S'=[S, S]$ is nilpotent, therefore   $S'$ is abelian
 by the previous case and $S'$ is an irreducible $L_1$-module by Cases 1 and  2. Since $\dim S/S' \leq 2$ by
 Corollary \ref{quotient} we have a direct decomposition $S=S' \oplus S_2$ of  $L_1$-submodules with $\dim S_{2}\leq 2.$
First suppose  that $\dim S_{2}=2.$  Let us show that $S_{2}$ is an irreducible $L_{1}$-module. Indeed, in other case $S_{2}\subseteq C_{S}(L_{1})$ and the centralizer $C_{S}(L_{1})$  a submodule of the  $L$-module $S.$  Because of previous cases we can assume that $\dim S'\geq 2$ and hence $S'$ is an irreducible $L_{1}$-module. Then obviously $C_{S}(L_{1})=S_{2}.$  Since $C_{S}(L_{1})=S_{2}$ is a subalgebra of $L$ the sum $S_{2}+L_{1}$ is a subalgebra of $L.$ The latter  is impossible because the subalgebra $S_{2}+L_{1}$ does not exist by the Case 1. This contradiction shows that $S_{2}$ is an irreducible $L_{1}$-module.

 Choose any nonzero elements $D_{2 }\in N_{+}$ and $h \in H$ and take standard bases $\lbrace e_0, e_1 \rbrace \subset  S_2$
 and $\lbrace f_0, f_1, \dots , f_m \rbrace \subset  S'$ of the $L_{1}$-modules $S_{2}$ and $S'$ respectively (recall that $L_{1}\simeq sl_{2}(\mathbb K)$).  Then the  linear operator  $\ad h$  has eigenvalues
 $1, -1$ on $S_{2}$. If the eigenvalues of $\ad h$ on $S'$ are even, then the elements
 $[e_i, f_j]$ are eigenvectors for $\ad h$  with odd eigenvalues. Since $[e_i, f_j]\in S'$ we see that $[e_i, f_j]=0$.
 Let now the eigenvalues of $\ad h$ on $S'$ be odd. Then $[e_i, f_j]$ are eigenvectors for
 $\ad h$ with even eigenvalues, so $[e_i, f_j]=0$, $i=0,1$, $j=0, 1, \dots m$.
 As $S'$ is abelian the latter means that $S' \subset Z(S)$. This is impossible because of our
 assumption on  $S$ and therefore $\dim S/S'=1$. Hence    $\dim S_{2}=1$.
 The subalgebra $S_{2}+L_{1}$ is obviously isomorphic  to $gl_{2}(\mathbb K )$ and $S'$ is an irreducible $S_{2}+L_{1}$-module. Since $S_{2}$ lies in the center of $S_{2}+L_{1}$ and $S$ is nonabelian we see that each nonzero element of $S_{2}$ acts on $S'$  as multiplication by a nonzero scalar.
We get a  Lie algebra of type 5 from this Lemma.

 \end{proof}
\begin{remark}\label{rem4}
 For each   type of  Lie algebras from this Lemma one can easily point out its  realization:

 (1) $\langle \frac{\partial}{\partial x}, \frac{\partial}{\partial y}, x\frac{\partial}{\partial x},
 x\frac{\partial}{\partial y}, y\frac{\partial}{\partial x}, y\frac{\partial}{\partial y},
 x(x\frac{\partial}{\partial x} +y\frac{\partial}{\partial y}), y(x\frac{\partial}{\partial x} +y\frac{\partial}{\partial y}) \rangle
 \simeq sl_3(\mathbb{K})$;

 (2)  $\langle \frac{\partial}{\partial x}, -x^2\frac{\partial}{\partial x}, -2x\frac{\partial}{\partial x}\rangle $
 $ \simeq sl_2(\mathbb{K})$
  and  $\langle \frac{\partial}{\partial x}, -x^2\frac{\partial}{\partial x} -2x\frac{\partial}{\partial x},
 \frac{\partial}{\partial y}, -y^2\frac{\partial}{\partial y}, -2y\frac{\partial}{\partial y}\rangle$
$ \simeq sl_2(\mathbb{K}) \oplus sl_2(\mathbb{K})$;

 (3) $\langle   x\frac{\partial}{\partial y}, y\frac{\partial}{\partial x},  x\frac{\partial}{\partial x}-y\frac{\partial}{\partial y},
 x^m(x\frac{\partial}{\partial x} +y\frac{\partial}{\partial y}), x^{m-1}y(x\frac{\partial}{\partial x} +y\frac{\partial}{\partial y}),
 \ldots , y^m(x\frac{\partial}{\partial x} +y\frac{\partial}{\partial y})
 \rangle\simeq sl_2(\mathbb K) \rightthreetimes V_m$.

 (4) $\langle  x\frac{\partial}{\partial x},
 x\frac{\partial}{\partial y}, y\frac{\partial}{\partial x}, y\frac{\partial}{\partial y},
 x^m(x\frac{\partial}{\partial x} +y\frac{\partial}{\partial y}), x^{m-1}y(x\frac{\partial}{\partial x} +y\frac{\partial}{\partial y}),
 \ldots , y^m(x\frac{\partial}{\partial x} +y\frac{\partial}{\partial y})
 \rangle\simeq gl_2(\mathbb K) \rightthreetimes V_m$.
 \end{remark}

We give a description of finite dimensional subalgebras of the Lie algebra $\widetilde{W}_2(\mathbb{K})$
 up to isomorphism as  Lie algebras.  In fact we give more information about such Lie algebras
 (up to choice of basis $\lbrace D_1, D_2 \rbrace$ of the two-dimensional vector space $\widetilde{W}_2(\mathbb{K})$
  over the field $R=\mathbb{K}(x,y)$). In order to clarify the structure of
 described subalgebras of $\widetilde{W}_2(\mathbb{K})$ we formulate the main Theorem in
 terms of generators and relations.
 \begin{theorem}\label{main}
 Let $L$ be a nonzero finite dimensional subalgebra of the Lie algebra $\widetilde{W}_2(\mathbb{K})$.
 Then the algebra $L$ belongs to one of the following types:

 {\rm (1)} $L=\langle e_1,\dots, e_n \rangle$, where $[e_i,e_j]=0$, $i,j=1 \dots n.$

 {\rm (2)} $L=\langle e_1,\dots, e_n , f\rangle$, where $[e_i,e_j]=0, [f, e_{i}]=e_{i}$, $i=1 \dots n.$

{\rm  (3)} $L=\langle e_0,\dots, e_n,f \rangle$, where $[e_i,e_j]=0$, $i,j=0 \dots n$, $[f,e_0]=\lambda e_0$,
 $[f,e_i]=\lambda e_i+e_{i-1}$, $i=1 \dots n$, $\lambda=0$ or $\lambda =1.$

 {\rm (4)} $L=\langle e_0,\dots, e_n,f \rangle$, where $[e_i,e_j]=0$, $i,j=0 \dots n$,
 $[f,e_i]=(1 +\beta m_i) e_i $, $i=0 \dots n$, $m_i \in \mathbb{Z}$, $\beta \in \mathbb{K}^{\star}$ and
 $m_i \neq m_j$ provided  that  $i \neq j.$

 {\rm (5)} $L=\langle e_0,\dots, e_n,f,g \rangle$, where $[e_i,e_j]=0$, $i,j=0 \dots n$, $[f,e_0]=0$,
 $[f,e_i]=e_{i-1}$, $i=1 \dots n$, $[g,e_i]=(i-\beta) e_i$, $i=0 \dots n$, $[g,f]=f-\gamma e_n$,
 $\beta, \gamma \in \mathbb{K}.$

 {\rm (6)} $L=\langle e_0,\dots, e_n,f,g \rangle$, where $[e_i,e_j]=0$, $i,j=0 \dots n$,
 $[f,e_i]=e_i$, $i=0 \dots n$, $[g,e_0]=0$, $[g,e_i]=e_{i-1}$, $i=1 \dots n$,
 $[f,g]=0.$

{\rm  (7)} $L=\langle e_0,\dots, e_n,f,g \rangle$, where $[e_i,e_j]=0$, $i,j=0 \dots n$,
 $[f,e_i]=e_i$, $i=0 \dots n$, $[g,e_i]=(1+\beta m_i) e_i$, $i=0 \dots n$, $[g,f]=0$,
 $\beta \in \mathbb{K}^{\star}$, $m_i \in \mathbb{Z}$, and
 $m_i \neq m_j$ if $i \neq j.$

{\rm  (8)} $L=\langle e_0,\dots, e_n,f,g,h \rangle$, where $[e_i,e_j]=0$, $i,j=0 \dots n$, $[f,e_0]=0$,
 $[f,e_i]=e_{i-1}$, $i=1 \dots n$, $[g,e_i]=e_i$, $i=0 \dots n$, $[g,f]=\alpha  e_n$, $[h, e_{i}]=-(\beta +i)e_{i},$
 $[h, f]=f-\gamma e_{n},$ $[h, g]=0, \alpha , \beta \in \mathbb{K}, \gamma =\alpha (\beta -n).$

{\rm  (9)}  $L \simeq sl_2(\mathbb{K})$, or  $L \simeq sl_2(\mathbb{K}) \oplus sl_2(\mathbb{K}) $;

{\rm  (10)}  $L \simeq sl_3(\mathbb{K})$;

 (11) $sl_2(\mathbb{K}) \rightthreetimes V_m$, where $V_m$ is the irreducible module
 over $sl_2(\mathbb{K})$ of dimension $m+1$, $m=0, 1, \dots$;

 (12) $gl_2(\mathbb{K}) \rightthreetimes V_m$, where $V_m$ is the irreducible module
 over $gl_2(\mathbb{K})$ of dimension $m+1$, $m=0, 1, \dots$ and  nonzero central
 elements of $gl_2(\mathbb K)$ act on $V_m$ as nonzero scalars.
 \end{theorem}

 \begin{proof}
 Let $L$ be a finite dimensional solvable subalgebra of the Lie algebra $\widetilde{W}_2(\mathbb{K})$.
 If $L$ is of rank 1 over $R$, then $L$ is of type 1 or  2  by Lemma \ref{rank1}.
 Let  $L$ be  of rank 2 over $R.$ If $L$ possesses an abelian ideal $I$ of rank 1 over $R$ which is maximal with this property, then $L$ is of type 3, 4 or 5 by Lemma \ref{solv1} (we denote $e_{i}=a_{i}D_{1}$ in type 4 and
  $e_{i}=(a^{i}/i!)D_{1}$ for types 3 and 5).  Let the ideal $I$ be nonabelian. Then by Lemma \ref{solv2} $L$ is one of types 6, 7 or 8 (as above we denote $e_{i}=a_{i}D_{1}$ in type 7 and
  $e_{i}=(a^{i}/i!)D_{1}$ for types 6 and 8, $f=bD_{1}$ for types 6 and 7 and $f=D_{2}, g=(v-\alpha (a^{n}/n!))D_{1}$ for type 8 of this Theorem).
  Further, let  $L$  be nonsolvable. If $L$ is semisimple, then $L$ is one of types 9 or 10 by Lemma \ref{semisimple}.
  Finally, if solvable radical of $L$ is nonzero, then $L$ is either of type 11 or of type 12 by Lemma  \ref{unsolvable}.

 \end{proof}


%
\end{document}